\documentclass[12pt, a4paper]{article}
\usepackage{amsmath, amssymb, amsthm}
\usepackage[english]{babel}
\usepackage[T1]{fontenc}

\newcommand{\be}{\begin{equation}}
\newcommand{\ee}{\end{equation}}
\newcommand{\bd}{\begin{displaymath}}
\newcommand{\ed}{\end{displaymath}}
\newcommand{\ba}{\begin{eqnarray}}
\newcommand{\ea}{\end{eqnarray}}

\def\R{{I \!\! R}}

\def\mD{ {\mathcal {D}}}

\def\f{\hat f}

\def\f{\varphi}

\def\e{\epsilon}

\def\v12{(v-w)}

\def\({\left(}
\def\){\right)}

\def\bgr#1\egr{{\allowdisplaybreaks\begin{gather}#1\end{gather}}}
\def\bma#1\ema{{\allowdisplaybreaks\begin{align}#1\end{align}}}

\def\oplem#1{\begin{lemma}\, {\rm #1}\, \it }
\def\cllem{\end{lemma}\rm \par }
\def\opthm#1{\begin{theorem}\, {\rm #1}\, \it }
\def\clthm{\end{theorem}\rm \par }

\def\N{\mathbb{N}}

\def\N{\mathbb{N}}
\def\R{\mathbb{R}}

\newcommand{\fer}[1]{(\ref{#1})}
\newcommand{\bq}{\begin{equation}}
\newcommand{\eq}{\end{equation}}
\def\bqa{\begin{eqnarray}}
\def\eqa{\end{eqnarray}}
\def\bd{\begin{displaymath}}
\def\ed{\end{displaymath}}

\renewcommand{\(}{\left(}
\renewcommand{\)}{\right)}

\usepackage{color}

\def\hg{\widehat g}

\def\hf{\widehat f}
\def\ff{\widehat f}

\newtheorem{thm}{Theorem}
\newtheorem{cor}[thm]{Corollary}
\newtheorem{lem}[thm]{Lemma}

\theoremstyle{remark}

\theoremstyle{definition}

\newenvironment{equations}{\equation\aligned}{\endaligned\endequation}
\def\derpar#1#2{\frac{\partial#1}{\partial#2}}

\usepackage{amssymb}

\title{Entropy inequalities for stable densities\\ and strengthened central limit  theorems}

\author{Giuseppe Toscani\thanks{Department of Mathematics, University of Pavia, via Ferrata 1, 27100 Pavia, Italy.
 }}

\begin{document}





\maketitle

\begin{abstract}
We consider the central limit theorem for stable laws in the case of the standardized sum of independent and identically distributed random variables with regular probability density function. By  showing decay of different entropy functionals along the sequence we prove convergence with explicit rate in various norms to a L\'evy centered density of parameter $\lambda >1$ . This introduces a new information-theoretic approach to the central limit theorem for stable laws, in which the main argument is shown to be the relative fractional Fisher information, recently introduced in \cite{To-Fish}. In particular, it is proven that, with respect to the relative fractional Fisher information, the L\'evy density satisfies an analogous of the logarithmic Sobolev inequality, which allows to pass from the monotonicity and decay to zero of the relative fractional Fisher information in the standardized sum to the decay to zero in relative entropy with an explicit decay rate.
\end{abstract}

{\bf{keyword}}
Central limit theorem; Fractional calculus; Shannon entropy; Fisher information; Information inequalities; Stable laws.





\section{Introduction}

Let us consider random variables $X_j$'s which are independent copies of a centered random variable $X$ which lies in the domain of normal attraction of a random variable
$Z_\lambda$ with {\it L\'evy symmetric stable distribution $\omega$}. Then, the central limit theorem for stable laws implies that the law $f_n$ of the normalized sums
 \be\label{stab}
T_n = \frac 1{n^{1/\lambda}}\sum_{j=1}^n X_j
 \ee
converges weakly to the law  of the centered stable law $Z_\lambda$, as $n$ tends to
infinity \cite{Fel, GK, LR79}. 
The density $\omega$ of a L\'evy symmetric random variable $Z_\lambda$  of order $\lambda$ is explicitly expressed in Fourier transform by the formula
 \be\label{levy}
 \widehat \omega(\xi) = \e^{-|\xi|^\lambda}.
 \ee 
In a recent paper \cite{To-Fish}, in connection with the study of the monotonicity (with respect to $n$) of entropy functionals of the normalized sums \fer{stab} we  introduced the definition of relative (to $Z_\lambda$) fractional Fisher information of a random variable $X$.
For a given random variable  $X$ in the domain of attraction of $Z_\lambda$, with $1<\lambda<2$, the relative (to $Z_\lambda$) fractional Fisher information of $X$ is expressed by the formula
 \be\label{fish-stable}
I_\lambda(X| \, Z_\lambda) =  I_\lambda(f|\, \omega) = \int_{\{f>0\}} \left( \frac{\mD_{\lambda-1} f(x)}{f(x)} - \frac{\mD_{\lambda-1} \omega(x)}{\omega(x)} 
 \right)^2f(x)\, dx,
 \ee
 where $f$ and $\omega$ denote the densities of $X$ and $Z_\lambda$,  respectively, and $\mD_{\nu}f(x)$, $0<\nu<1$  is the fractional derivative of order $\nu$ of $f(x)$ (cf. the Appendix for the definition).  Note  that the relative fractional Fisher information is obtained from the classical one, expressed by
\be\label{rel-f}
 I(X|\, Z_\lambda ) = I(f|\, \omega) = \int_{\{f>0\}} \left(\frac{f'(x)}{f(x)} - \frac{\omega'(x)}{\omega(x)}\right)^2 f(x)\, dx,
 \ee
by substituting the standard derivative with the fractional derivative of order $\lambda -1$, which is such that $0<\lambda -1 < 1$ for $1 <\lambda <2$. 
 This nonlocal functional is based on a suitable modification of the  linear score function used in theoretical statistics. As the
linear score function $f'(X)/f(X)$ of a random variable $X$ with a (smooth) probability density $f$ identifies  Gaussian variables  as the
unique random variables for which the linear score is proportional to $X$ (i.e.$f'(X)/f(X) =CX)$ , L\'evy symmetric stable laws are now identified as the unique random variables $Y$ for which the new defined linear fractional score is proportional to $Y$ (cf. Section \ref{score}). Consequently, the relative (to $Z_\lambda$) fractional Fisher information \fer{fish-stable} can be equivalently written as
 \be\label{normal}
 I_\lambda(X| \, Z_\lambda) =  \int_{\{f>0\}} \left( \frac{\mD_{\lambda-1} f(x)}{f(x)} +\frac x\lambda  \right)^2f(x)\, dx,
 \ee

This analogy was pushed further to show that the relative fractional Fisher information \fer{fish-stable},  satisfies almost all properties of the classical relative Fisher information \fer{rel-f}.

The results in \cite{To-Fish},  include both monotonicity  of the relative fractional Fisher information along the sequence  $T_n$,  so that $I_\lambda(T_{n+1}|\, Z_\lambda) \le I_\lambda(T_n|\, Z_\lambda)$ for all $n \ge 1$, and an explicit rate of decay to zero
 \be\label{giu2}
I_\lambda(T_n|\, Z_\lambda) \le \left( \frac 1n \right)^{(2-\lambda)/\lambda} I_\lambda(X|\, Z_\lambda).
 \ee
These properties make evident that this new concept is quite useful to extract most of the statistical properties of the densities $f_n$ of the sequence $T_n$. 

However, at difference with what happens in the case of the classical central limit theorem, some relevant problems remained open. Among other questions,  both  the (eventual) monotonicity of the relative (to $Z_\lambda$) Shannon entropy, defined by
 \be\label{ren}
  H(X|\, Z_\lambda ) = H(f|\, \omega) = \int_{\R} f(x) \log \frac{f(x)}{\omega(x)}\, dx,
 \ee
along the sequence $f_n$  and  its decay to zero as $n$ converges towards infinity were left untouched.

At difference with the classical central limit theorem,  a prominent role is here played by the domain of attraction, which selects the laws of the random variable $X$ which eventually imply convergence of the law of the sum \fer{stab} towards $\omega$. The necessity to start sufficiently close to the target density, introduces additional difficulties in extending entropy arguments, and requires the development of various \emph{ad hoc} techniques,  usually based on the concept of relative entropy.  

Convergence to stable laws by means of the standard relative  Shannon entropy  \fer{ren} and  relative Fisher information \fer{rel-f} has been recently investigated by Bobkov,  Chistyakov  and  G\"otze   \cite{BCG-en, BCG3} (cf. also \cite{BCG1, BCG2}). The main result in \cite{BCG-en, BCG3}  was to show that, assuming the weak convergence of the normalized sums $T_n$ defined in \fer{stab} to a random variable $Z_\lambda$ with a non-extremal stable law with $0<\lambda <2$, then the relative Shannon entropy  $H(T_n|\, Z_\lambda)$ (respectively the Fisher information $I(T_n|\, Z_\lambda)$) converges to zero as time goes to infinity, if and only if $H(T_n|\, Z_\lambda) < +\infty$ (respectively $H(T_n|\, Z_\lambda) < +\infty$) for some $n >0$.  These results, however, do not contain information on the time decay of the relative Shannon entropy and the relative Fisher information. 

While the reading of \cite{BCG-en, BCG3} makes it clear that both the relative Fisher information \fer{rel-f} and the relative entropy \fer{ren} allow to get similar results, it appears also evident that the fractional Fisher information introduced in \cite{To-Fish} is more adapted to the study of the problem of convergence towards a stable law, in that it furnishes both  monotonicity along the sequence \fer{stab} and an explicit rate of convergence as in \fer{giu2}.  Hence, in view of the strongness of the results that can be obtained via the fractional Fisher information, it would be desirable to establish a connection of this new functional with the standard relative Shannon entropy.

To  give an answer to this question,  in analogy with the classical case, where the connection between relative entropy and relative Fisher information can be established by studying the time-evolution of the relative entropy of the solution to the Fokker--Planck equation \cite{AMTU, Tos3, Tos4}, in what follows we will investigate the time-evolution of the relative Shannon entropy along the solution of a suitable Fokker--Planck equation with fractional diffusion
 \be\label{FFPP}
\derpar{f}{t} = \derpar{\null}{x}\left(D_{\lambda-1} f + \frac x\lambda f\right) ,  
 \ee
where $1<\lambda < 2$, and the initial datum $\varphi(x)$ belongs to the domain of normal attraction of the L\'evy stable law $\omega$ of parameter $\lambda$, as given by \fer{levy}, which results to be a stationary solution of  equation \fer{FFPP}.  Fractional diffusion equations are well-studied, since they result quite useful in the description  of many physical processes, including turbulent flows \cite{Shl}, diffusion in complex systems \cite{Ott}, chaotic dynamics of classical conservative systems \cite{SZK93,KZB}, and
others. Also, mathematical aspects of fractional diffusion equations have been recently investigated from the point of view of mass transportation techniques in \cite{Erb}, and their connection with non-local kinetic equations of Boltzmann-type have been studied in \cite{FPTT}.

As we will see, this Fokker--Planck equation relates in a clear way the relative entropy to the relative fractional Fisher information, and allows to recover, similarly to what happens for the classical Fokker--Planck equation \cite{AMTU, Tos3, Tos4} an inequality similar to the classical logarithmic Sobolev inequality. The new inequality bounds the relative (to $Z_\lambda$) Shannon entropy in terms of the relative (to $Z_\lambda$) fractional Fisher information and the standard Fisher information  
 \be\label{Sobo1}
 H(X |\, Z_\lambda) \le \lambda \, 2^{1/\lambda} \, \min\{ I( X),\, I(Z_\lambda) \}^{1/2}\, I_\lambda(X|\, Z_\lambda)^{1/2}. 
 \ee
As for the Gaussian case, inequality \fer{Sobo1} allows to recover convergence results and rates of convergence for the relative Shannon entropy via convergence results and rates for the fractional Fisher information.  Inequality \fer{Sobo1} will be proven in Section \ref{entropies}. Then, we will pass from convergence in relative entropy to $L^1(\R)$-convergence. Indeed, convergence in relative entropy of the sequence $f_n$ to $\omega$ at the rate $t^{-\mu}$ implies convergence in $L^1(\R)$ of $f_n$ to $\omega$ at the sub-optimal rate $t^{-\mu/2}$ by Csiszar--Kullback inequality \cite{Csi, Kul}.

By using $L^1(\R)$ convergence, we will subsequently prove convergence in various Sobolev spaces at an explicit rate, which depends on the (increasing) regularity of $f_n$ as $n $ increases. This will be shown in Section \ref{regolare}. 

Our results are largely inspired by the treatment of the analogous problems in the case of the  the standard central limit theorem. There,  starting from the pioneering work of Linnik   \cite{Lin}, who first used Fisher information in a proof of the central limit theorem,  entropy functionals, in particular Shannon entropy and Fisher information,  have successfully been used to quantify the change in entropy as a result of convolution.   Let us briefly recall these results.

For $j  \in \N$ , $j \ge 1$ let the $X_j$'s be independent copies of a centered random variable $X$ with variance
$1$.  Then the (classical) central limit theorem implies that the law of the normalized
sums
 \[
S_n = \frac 1{\sqrt n}\sum_{j=1}^n X_j
 \]
converges weakly to the law of the centered standard Gaussian $Z$, as $n$ tends to infinity.

This result has a clear interpretation in terms of statistical mechanics.
Indeed, consider the entropy functional (or Shannon entropy) of a real valued random variable $X$
with density $f$,  defined as
 \be\label{Shan}
H(X) = H(f) = - \int_{\R} f(x) \log f(x)\, dx,
 \ee
provided that the integral makes sense. Among random variables with the same variance
$\sigma$ the standard Gaussian $Z$ with variance $\sigma$ has the largest entropy. This property suggests that the entropy could increase along the sequence $S_n$, in order to reach its maximal value. 
A direct consequence of the entropy power inequality, postulated by Shannon \cite{Sha}
in the fourthies, and subsequently proven by Stam  \cite{Sta} (cf. also Blachman
\cite{Bla}),  implies that $H(S_2)\ge H(S_1)$. The entropy of the normalized sum of two
independent copies of a random variable is larger than that of the original. A shorter
proof was obtained later by Lieb \cite{Lieb} (cf. also  \cite{Bar, Joh, JB} for
exhaustive presentation of the subject). While inductively expected that  the entire
sequence $H(S_n)$ should increase with $n$, as conjectured  by Lieb in $1978$
\cite{Lieb}, a rigorous  proof of this result was found only $25$ years later by
Artstein, Ball,  Barthe and A. Naor  \cite{ABBN1, ABBN2}.

Other simpler proofs of the monotonicity of the sequence $H(S_n)$ have been recently
obtained by Madiman and Barron \cite{MB, BM} and Tulino and Verd\'u \cite{TV}.

Most of the results about monotonicity  benefit from the reduction from entropy to another information-theoretic notion, the Fisher information of a random variable $X$ with a (smooth) density $f$, defined as
 \be\label{fish}
I(X) = I(f) = \int_{\{f>0\}} \frac{|f'(x)|^2}{f(x)} \, dx.
 \ee
Among random variables with the same variance $\sigma$, the Gaussian $Z$ has smallest Fisher
information $1/\sigma$.

Fisher information and entropy are related each other by the so-called
de Bruijn relation \cite{Bar, Sta}. If  $u(x,t)= u_t(x)$ denotes the solution to the initial value problem for the heat equation in the whole space $\R$,
 \be\label{heat}
 \frac{\partial u}{\partial t} =  \frac{\partial^2 u}{\partial x^2},
 \ee
 leaving from an initial probability density function $f(x)$,
 \be\label{deb}
 I(f) = \frac d{dt} H(u_t)|_{t=0}.
 \ee
A particularly clear explanation of this link is given in the article of Carlen and Soffer \cite{CS} (cf. also Barron \cite{Bar} and Brown \cite{Bro}).

De Bruijn equality first outlines that diffusion equations play an important role both in connecting entropy functionals, and in deriving  inequalities, a role which is nowadays well understood \cite{joint}. 

It is remarkable, however,  that convergence in relative (to the Gaussian $Z$) Fisher information with an explicit rate (the analogous of formula \fer{giu2}) has been shown to hold only for random variables $X$ with an absolutely continuous density with finite restricted Poincar\'e constant \cite{JB}.

In more details, in Section \ref{score} we will recall the main properties of the standard and fractional Fisher information, with a short explanation of the results obtained in \cite{To-Fish}. Then the connection between the relative Shannon entropy and the relative fractional Fisher information via the study of the entropy decay of the solution to the   fractional Fokker--Planck equation \fer{FFPP} will be developed in Section \ref{entropies}. Last, further results about regularity and convergence of the sequence $T_n$ defined in \fer{stab} will be studied in Section \ref{regolare}. We postpone to an Appendix the principal facts about fractional derivatives, functional spaces and some well-known densities that belong to the domain of attraction of L\'evy density.

\section{Fisher and fractional Fisher information}\label{score}

In the rest of this paper, if not explicitly quoted, and without loss of generality,
we will always assume that any random variable $X$ we will consider is centered, i.e.
$E(X)=0$, where as usual $E(\cdot)$ denotes mathematical expectation. We will start this section by recalling various well-known properties about the classical Fisher and relative Fisher information which will be used in the rest of this paper. Further results are collected in \cite{BCG1}.

The change of Fisher information when applied to convolutions is quantified by the well-known Stam's Fisher information inequality \cite{Sta}, which gives a
lower bound on the inverse of Fisher information of the sum
of independent random variables with (smooth) densities
 \be\label{f-stam}
 \frac 1{I(X+Y)} \ge \frac 1{I(X)} + \frac 1{I(Y)},
 \ee
with equality if and only X and Y are Gaussian random variables with proportional variances. 
A direct consequence of this inequality is that
 \be\label{min2}
 I(X+Y) \le \min \{I(X), I(Y)\},
 \ee 
if $X$ and $Y$ are independent random variables. 

Fisher information can also be described in the language of theoretical statistics. There, the score or efficient score \cite{CH, BM} is the
derivative, with respect to some parameter $\theta$, of the logarithm of the
likelihood function (the log-likelihood). If the observation is $X$ and its likelihood
is $L(\theta;X)$, then the score $\rho_L(X)$ can be found through the chain rule
 \be
\rho_L(\theta, X) = \frac{1}{L(\theta;X)} \frac{\partial L(\theta;X)}{\partial\theta}.
\ee Thus the score indicates the sensitivity of $L(\theta;X)$ (its derivative
normalized by its value). In older literature, the term \emph{linear score} refers to
the score with respect to an infinitesimal translation of a given density. In this
case, the likelihood of an observation is given by a density of the form
$L(\theta;X)=f(X+\theta)$. According to this definition, given a random variable $X$
in $\R$ distributed with a differentiable probability density function $f(x)$, its
linear score $\rho$ (at $\theta=0$) is given by
 \be\label{sco1}
\rho(X) = \frac{f'(X)}{f(X)}.
 \ee
The linear score has zero mean, and its variance is just the Fisher information
\fer{fish} of $X$.

Also, the notion of relative score has been recently considered in information theory
\cite{Guo} (cf. also \cite{BCG3}).  For every pair of random variables $X$ and $Y$
with differentiable density functions $f$ (respectively $g$), the score function of
the pair relative to $X$ is represented by
 \be\label{rel-sco}
\tilde\rho(X) = \frac{f'(X)}{f(X)} - \frac{g'(X)}{g(X)}.
 \ee
In this case, the relative (to $X$) Fisher information between $X$ and $Y$ is just the
variance of $\tilde\rho(X)$. This notion  is satisfying because it represents the
variance of some error due to the mismatch between the prior distribution $f$ supplied
to the estimator and the actual distribution $g$. Obviously, whenever $f$ and $g$ are
identical, then the relative Fisher information is equal to zero.

The relative Fisher information has been recently used in entropic proofs of the central limit theorem both in the classical case \cite{BCG2}, and in the case of stable laws \cite{BCG3}. 

In the classical case the relative (to the Gaussian) Fisher information takes a simple form, in view of the properties of the Gaussian density. Indeed, let $\omega_\sigma(x)$ denote the Gaussian density in $\R$ with zero mean and variance
$\sigma$
 \be\label{max}
\omega_\sigma(x) = \frac 1{\sqrt{2\pi \sigma}}\exp\left(- \frac{|x|^2}{2\sigma}\right).
 \ee
Then a Gaussian random variable of density $z_\sigma$ is uniquely defined by a linear score
function
 \[
 \rho(Z_\sigma) = - Z_\sigma/\sigma.
 \]
Also, the relative (to $X$) score function of $X$ and $Z_\sigma$ takes the simple
expression
 \be\label{rel-z}
\tilde\rho(X) = \frac{f'(X)}{f(X)} + \frac X\sigma,
 \ee
which induces a (relative to the Gaussian) Fisher information
 \be\label{fish-r}
\tilde I(X) = \tilde I(f) = \int_{\{f>0\}} \left( \frac{f'(x)}{f(x)} + \frac x\sigma
\right)^2f(x)\, dx.
 \ee
Clearly, $\tilde I(X) \ge 0$, while $\tilde I(X) = 0$ if $X$ is a centered Gaussian
variable of variance $\sigma$.

Having in mind this expression, the concept of Fisher information has been extended in \cite{To-Fish} to cover fractional derivatives (cf. the Appendix for the definition).  Given a random variable $X$ in $\R$ distributed with
a probability density  function $f(x)$ that has  a well-defined  fractional derivative
of order $\nu$, with $0<\nu < 1$, the linear fractional score, denoted by
$\rho_{1+\nu}$ is given by
 \be\label{sco2}
\rho_{1+\nu}(X) = \frac{\mD_\nu f(X)}{f(X)}.
 \ee
Thus the linear fractional score indicates the non local (fractional) sensitivity of
$f(X+\theta)$ at $\theta =0$ (its fractional derivative normalized by its value).
Analogously to the classical case,  the linearity of the fractional score of $X$ identifies $X$ as a L\'evy distribution of order $1+ \nu$ (cf. the Appendix). In fact,  the L\'evy random
variable of parameter $\lambda$ defined in \fer{levy}, with $1<\lambda<2$, is uniquely defined by the linear fractional score function
 \be\label{zz}
 \rho_\lambda (Z_{\lambda}) = -\frac{Z_{\lambda}}{\lambda}.
 \ee
It is important to remark that, at difference with the case of the standard linear score, the variance of the fractional score is in general
unbounded. One can easily realize this by looking at the variance of the
fractional score in the case of a L\'evy variable. For a L\'evy variable, in fact, the
variance of the fractional score coincides with a multiple of its variance, which is
unbounded \cite{GK,LR79}. For this reason, a consistent definition in this case is
represented by the relative fractional score. 

The relative (to $Z_\lambda$) fractional score function of $X$  assumes the
simple expression
 \be\label{rel-z2}
\tilde\rho_\lambda(X) = \frac{\mD_{\lambda-1} f(X)}{f(X)} + \frac X{\lambda},
 \ee
which induces a (relative to the L\'evy) fractional Fisher information (in short
$\lambda$-Fisher relative information)
 \be\label{fish-r2}
I_\lambda(X) =  I_\lambda(f) = \int_{\{f>0\}} \left( \frac{\mD_{\lambda-1} f(x)}{f(x)} +
\frac x\lambda \right)^2f(x)\, dx.
 \ee
The fractional Fisher information is always greater or equal than zero, and it is
equal to zero if and only if $X$ is a L\'evy symmetric stable distribution of order
$\lambda$. We remark that, due to the fat tails of the L\'evy density, $I_\lambda(X)$ is bounded only if the random variable $X$ has a
probability density function which is suitably closed to the L\'evy stable law
(typically lies in a subset of the domain of attraction). 

We will define by $\mathcal
P_\lambda$ the set of probability density functions such that $I_\lambda(f) <
+\infty$, and we will say that a random variable $X$ lies in the domain of attraction
of the $\lambda$-Fisher information if $I_\lambda(X) < +\infty$. More in general, for
a given positive constant $\upsilon$, one could consider other relative fractional score
functions given by
 \be\label{rel-z3}
\tilde\rho_{\lambda, \upsilon}(X) = \frac{\mD_{\lambda-1} f(X)}{f(X)} + \frac
X{\lambda\upsilon}.
 \ee
This leads to the relative fractional Fisher information
 \be\label{fish-r3}
 I_{\lambda,\upsilon}(X) =  I_{\lambda,\upsilon}(f) = \int_{\{f>0\}} \left( \frac{\mD_{\lambda-1}
f(x)}{f(x)} + \frac x{\lambda\upsilon} \right)^2f(x)\, dx.
 \ee
Clearly,  $I_\lambda= I_{\lambda, 1}$. Analogously, we will define by $\mathcal
P_{\lambda,\upsilon}$ the set of probability density functions such that
$I_{\lambda,\upsilon}(f) < +\infty$, and we will say that a random variable $X$ lies
in the domain of attraction  if $I_{\lambda, \upsilon}(X) < +\infty$.

Note that the relative Fisher information $I_\lambda(X)$ and $I_{\lambda,\upsilon}(X)$ are related each other.  Indeed, for any given random variable $X$ such that one of the two sides is
bounded, and positive constant $\upsilon$, the following identity holds
 \be\label{scal}
 I_{\lambda,\upsilon}(\upsilon^{1/\lambda}X) = \upsilon^{-2(1-1/\lambda)} I_\lambda
 \left(X\right).
 \ee
The domain of attraction of
the relative fractional Fisher information is not empty, and contains probability densities which belong to the domain of attraction of $Z_\lambda$. As shown in the Appendix of \cite{To-Fish}, {\it Linnik distribution} \cite{L, L2} belongs to $\mathcal
P_\lambda$. Linnik distribution is  expressed in Fourier variable by
  \be\label{Max-f}
  \widehat p_\lambda(\xi) = \frac 1{1+ |\xi|^\lambda}.
  \ee
For all $0<\lambda \leq 2$,  the function  \fer{Max-f} is the
characteristic function of a symmetric probability distribution. In addition, when
$\lambda > 1$, $\widehat p_\lambda \in L^1(\R)$, which, by applying the
inversion formula, shows that $p_\lambda$ is a probability density function.
(cf. Kotz and Ostrovskii \cite{KO}).

The main properties of the relative fractional Fisher information have been obtained in \cite{To-Fish}.  Most of these properties are
analogous of the ones proven in \cite{BM} for the standard linear
score function and the classical Fisher information.  

The main result in \cite{To-Fish} is the fractional version of the Blachman--Stam inequality \fer{f-stam} \cite{Bla, Sta}, which allows to bound the relative fractional Fisher information of the sum of
independent variables in terms of the relative fractional Fisher information of its
addends. 

\begin{thm}\label{bl} Let $X_j$, $j_1,2$ be independent random variables such that their relative
fractional Fisher information  functions $I_\lambda(X_j)$, $j=1,2$ are bounded for
some $\lambda$, with $1<\lambda <2$. Then, for each constant $\varepsilon$ with $0<\varepsilon <
1$,  $I_\lambda(\varepsilon^{1/\lambda}X_1+(1-\varepsilon)^{1/\lambda}X_2) $ is bounded, and
 \be\label{BS1}
I_\lambda (\varepsilon^{1/\lambda}X_1+(1-\varepsilon)^{1/\lambda}X_2)\le \varepsilon^{2/\lambda}
I_\lambda\left(X_1\right) + (1- \varepsilon)^{2/\lambda} I_\lambda\left( X_2\right).
 \ee
Moreover, there is equality in \fer{BS1} if and only if,
up to translation, both $X_j$, $j=1,2$ are L\'evy variables of exponent $\lambda$.
\end{thm}

An important consequence of Theorem \ref{bl}, which will be widely used in the rest of the paper, is concerned with the form that takes inequality \fer{BS1} when one of the two variables involved is a L\'evy variable. For any given positive constant
$\varepsilon <1$, and random variable $X$ with density function $f$ let us denote by
$f_\varepsilon$ the density of the random variable $X_\varepsilon = (1-\varepsilon)^{1/\lambda}X + \varepsilon^{1/\lambda}Z$, where the symmetric stable L\'evy
variable $Z$ of order $\lambda$ is independent of $X$. Then \fer{BS1} takes the form
 \be\label{ep}
 I_{\lambda}(X_\varepsilon) \le (1-\varepsilon)^{2/\lambda} I_\lambda \left(X\right).
 \ee

In other words, the relative fractional Fisher information of the smoothed version $X_\varepsilon$
of the random variable  is always smaller than the relative fractional Fisher
information of $X$. Moreover,
 \[
 \lim_{\varepsilon \to 0} I_{\lambda}(X_\varepsilon) = I_\lambda \left(X\right).
 \]
A further result for the relative fractional Fisher information refers to its monotonicity along the normalized sums in the central limit theorem for stable laws. It holds \cite{To-Fish}

\begin{thm}\label{main3} Let $T_n$ denote the sum \fer{stab}, where the random variables $X_j$
are independent copies of a centered random variable $X$ with bounded relative
$\lambda$-Fisher information, $1<\lambda<2$. Then, for each $n >1$, the relative
$\lambda$-Fisher information of $T_n$ is decreasing in $n$, and the following bound
holds
 \be\label{mm}
I_\lambda\left(T_n \right)\le \left(\frac{n-1}n
\right)^{(2-\lambda)/\lambda}I_\lambda\left(T_{n-1} \right).
 \ee
\end{thm}

As pointed out in \cite{To-Fish}, at difference with the case of the standard central limit
theorem, where $\lambda = 2$ and the monotonicity result of the classical relative
Fisher information reads $I(S_n) \le I(S_{n-1})$, in the case of the central limit
theorem for stable laws, the monotonicity of the relative $\lambda$-Fisher information
also gives a rate of decay. Indeed, formula \fer{mm} of Theorem \ref{main3} shows that,
for all $n>1$
 \be\label{giu}
I_\lambda(T_n) \le \left( \frac 1n \right)^{(2-\lambda)/\lambda} I_\lambda(X),
 \ee
namely convergence in relative $\lambda$-Fisher information sense at rate
$1/n^{(2-\lambda)/\lambda}$.

\section{Further entropy inequalities}\label{entropies}

The goal of this Section is to present in more details the connection between the relative fractional Fisher information and fractional evolution equations of Fokker--Planck type. The connection between relative Shannon entropy, Fisher information and the Fokker--Planck equation, which leads to logarithmic Sobolev inequalities, is well-known, starting from \cite{AMTU, Tos3, Tos4}. We will adapt the method in \cite{Tos3} to the present situation. Given a random variable  $X$ of density $h(x)$, and a constant $a >0$, let us denote by $h_a(x)$ the probability density of $aX$. Let $Y$ a random variable with density $\varphi$, and let $Z_\lambda$ be a L\'evy variable independent of $Y$, such that $1<\lambda<2$, of density $\omega(x)$.  For a given $t >0$ we define
 \be\label{solFP}
X_t= \alpha(t)Y + \beta(t)Z_\lambda, 
 \ee
where 
\be\label{con22}
\alpha (t)= e^{-t/\lambda}, \qquad \beta(t)= (1-e^{-t})^{1/\lambda},
\ee
thus satisfying the relation
 \be\label{norm3}
 \alpha^\lambda(t) + \beta^{\lambda}(t) = 1.
 \ee
Then, the random variable $X_t$, $t >0$, has a density given by
 \be\label{slFP}
 f(x,t)= \varphi_{\alpha(t)}\ast  \omega_{\beta(t)}(x), 
\ee
where, as usual, $\ast$ denotes the convolution product. It can be verified directly by computations that $f(x,t)$ solves the  Fokker-Planck equation \fer{FFPP}, characterized by a fractional diffusion,
with initial value $f(x, t=0) = \varphi(x)$. 

Fokker-Planck type equations with fractional diffusion appear in
many physical contexts \cite{BWM,Cha98,KZB,MFL02}, and have been
intensively studied both from the modeling and the qualitative point
of view.
 Also, fractional diffusion equations in the nonlinear setting have
been introduced and successfully studied in the last years by Caffarelli, Vazquez et al. \cite{CV,
Car, Va1, Va2}. 

The stationary solution of the Fokker-Planck equation \fer{FFPP} is the L\'evy density  of order $\lambda$ defined by \fer{levy} in terms of its Fourier transform. Indeed, if $f(x,t)$  denotes the solution to equation \fer{FFPP},
passing to Fourier transform, we obtain that $\widehat f(\xi,t)$ solves the equation
 \be\label{FPF-fou}
\derpar{\ff}{t} = -|\xi|^\lambda \ff(\xi,t) - \frac\xi\lambda \derpar{\ff(\xi,t)}\xi,
 \ee
which clearly has a stationary solution given by the L\'evy density \fer{levy} of order $\lambda$.

 Integrating equation \fer{FPF-fou} along characteristics (cf. \cite{CT1} for the analogous derivation in the case of the classical Fokker--Planck equation), it is immediate to see that a solution $f(x, t)$ of 
 \eqref{FPF-fou} can be explicitly expressed by
\be\label{sol-FP-four}
  \widehat f(\xi, t)= \widehat \varphi\left( \xi e^{- t/\lambda }\right ) e^{-|\xi|^{\lambda}(1-e^{-t})},
 \ee
which is the Fourier transform of \fer{slFP}.  
Similarly to the  classical Fokker--Planck  equation, where the solution interpolates continuously between the initial datum and the Gaussian density, here the solution to the Fokker--Planck equation with fractional diffusion interpolates continuously between the initial datum $\varphi$ and the L\'evy density $L$ of order $\lambda$.
Also, formula \fer{sol-FP-four} shows  that the L\'evy density $\omega$ is invariant under scaled convolutions satisfying \fer{norm3}, so that
\be\label{scal3}
 \omega(x) = \omega_{\alpha(t)}\ast  \omega_{\beta(t)}(x),
 \ee
Let us consider  the relative (to the L\'evy density) entropy of $X_t$
 \be\label{rel-FP}
H(X_t|\, Z_\lambda) = H(f(t)|\, \omega)= \int_\R f(x,t) \log \frac{f(x,t)}{\omega(x)} \, dx.
 \ee
Our main goal here is to study the time-behavior of the relative entropy. The following holds

\begin{lem}\label{l11}
Let the density $\varphi$  be such that $H(\varphi|\, \omega)$ is finite. Then, if $f(x,t)$ is given by \fer{slFP}, the relative entropy $H(f(t)|\,\omega)$ is monotonically decreasing in time. In addition, if the density $\varphi$ belongs to the  domain of normal attraction of $Z_\lambda$, as time goes to infinity
 \[
 \lim_{t\to \infty} H(f(t)|\, \omega) = 0.
 \]
\end{lem}

\begin{proof}
Let $f,g$ be probability densities such that their relative entropy $H(f|g)$ is finite. Since the function 
\[
k(x,y) = x \log\frac xy, \quad x \ge 0, y \ge 0
\]
is jointly convex,  if  $\mu$ is a probability density, Jensen's inequality implies
 \[
 f\ast \mu(x) \log \frac{f\ast \mu(x)}{g\ast \mu(x)} \le \int_\R f(x-y)\log \frac{f(x-y)}{g(x-y)}\, \mu(y) \,dy.
 \] 
Hence, integrating on both sides, and applying Fubini's theorem one obtains the inequality
 \[
 H(f\ast\mu|\, g\ast\mu) \le H(f|\, g).
 \]
Let us apply this inequality to the  relative (to the L\'evy density) entropy of $f(x,t)$. Since the density $\omega$ is invariant under convolution scaled accordingly to $\alpha$ and $\beta$ (cf. formula \fer{scal3})
it holds
 \[
 H(f(t)|\, \omega) =  H(\varphi_{\alpha(t)}\ast  \omega_{\beta(t)}|\, \omega_{\alpha(t)}\ast  \omega_{\beta(t)}).
 \]
On the other hand, if $\mu(t) = \beta(t)/\alpha(t)$, elementary computations show that, for any given  $t >0$
 \begin{equation}
H(\varphi_{\alpha(t)}\ast  \omega_{\beta(t)}|\, \omega_{\alpha(t)}\ast  \omega_{\beta(t)}) = 
 H(\varphi \ast \omega_{\mu(t)} , \, \omega\ast \omega_{\mu(t)} ) \le 
 H(\varphi  | \, \omega ).
 \end{equation}
 This implies that the relative entropy of the solution at any time $t>0$ is less than the initial relative entropy. Now, consider that, for any time $t_1 >t$, by applying formula \fer{scal3} one get
  \[
 f(x,t_1)= \varphi_{\alpha(t_1)}\ast  \omega_{\beta(t_1)}(x) = \left[\varphi_{\alpha(t)}\ast  \omega_{\beta(t)}\right]_{\alpha(t_1-t)}\ast \omega_{\beta(t_1-t)}(x), 
  \]
 namely that the solution at time $t_1$ corresponding to the initial datum $\varphi$ is equal to the solution obtained at time $t_1-t$ corresponding to the initial datum  given by the solution at time $t$ which starts from the same initial datum $\varphi$.  
Hence the same monotonicity argument can be used starting from  any subsequent time. This shows that the relative entropy is non-increasing in time.
 
Clearly, monotonicity of the relative entropy does not imply that the relative entropy $H(f(t)|\, \omega)$ converges towards zero as time goes to infinity, even if the a.e. convergence to $\omega$ of $f(t)$ could suggest that this is the case.  In order to prove convergence to zero, we further choose the density $\varphi$ in the domain of normal attraction of the stable law $Z_\lambda$. 

Let us briefly recall some information about the domain of attraction of a stable law. More details can be found in the book
\cite{Ibra71} or, among others, in the papers \cite{BLM}, \cite{BLR}. A centered
distribution function $F$ belongs to the domain of normal attraction of the $\lambda$-stable
law \fer{levy} with distribution function $\omega(x)$ if and only if $F$ satisfies the conditions
$(-x)^\lambda F(x)\to c$  and $x^\lambda (1-F(x))\to c$ as $x\to
+\infty$ i.e.
\begin{equation}\label{dof}
\begin{aligned}
&F(-x)=\frac{c}{|x|^\lambda}+S_1(-x) \ \ \ \ \ \ {\rm and }\ \ \ \ \
\ \ 1-F(x)=\frac{c}{x^\lambda}+S_2(x) \ \ \ \ \ \ (x>0)\\
&S_i(x)=o(|x|^{-\lambda})\ \ \  \ {\rm as}\ |x|\to +\infty, \ \ \
i=1,2\\
\end{aligned}
\end{equation}
where $c=\frac{\Gamma(\lambda)}{\pi}\sin\left(\frac{\pi\lambda}{2}\right)$.

   Conditions \fer{dof} can be rephrased in terms of probability densities by saying that a centered
density $f$ belongs to the domain of normal attraction of the $\lambda$-stable
law \fer{levy} with distribution function $\omega(x)$ if and only if $f$ satisfies
$|x|^{\lambda +1}f(x)\to \lambda c$   as $x\to
\pm \infty$. In other words, if a density $f$ belongs to the domain of normal attraction of the $\lambda$-stable
law \fer{levy} , and it is bounded, also the function $|x|^{\lambda +1}f(x)$ is bounded. This property allows to show that, provided the initial density $\varphi$ belongs to the domain of normal attraction of $Z_\lambda$,  the solution to the Fokker--Planck equation belongs to the same domain of normal attraction   for all times $t>0$, and the following bound holds
 \be\label{attr}
 \lim_{x \to \pm \infty} |x|^{\lambda +1}f(x,t) \le 2^\lambda\, \lambda \, c.
 \ee
Indeed, for any $x,y \in \R$
\[
|x|^{\lambda +1} \le 2^\lambda \left( |x-y|^{\lambda +1} + |y|^{\lambda +1}\right),
\]
and
\[
\lim_{x \to \pm \infty} |x|^{\lambda +1}f(x,t) \le 2^\lambda \lim_{x \to \pm \infty} \int_\R |x-y|^{\lambda +1}\varphi_{\alpha(t)}(x-y) \omega_{\beta(t)}(y) \, dy \, +
\]
\[
2^\lambda \lim_{x \to \pm \infty} \int_\R |x-y|^{\lambda +1}\omega_{\beta(t)}(x-y) \varphi_{\alpha(t)}(y) \, dy  =
\]
\[
 2^\lambda \alpha(t)^\lambda \lim_{x \to \pm \infty} \int_\R \left|\frac{x-y}{\alpha(t)}\right|^{\lambda +1}\varphi\left(\frac{x-y}{\alpha(t)}\right) \omega_{\beta(t)}(y) \, dy \, +
\]
\[
 2^\lambda\beta(t)^\lambda \lim_{x \to \pm \infty} \int_\R \left|\frac{x-y}{\beta(t)}\right|^{\lambda +1}\omega\left(\frac{x-y}{\beta(t)}\right) \varphi_{\alpha(t)}(y) \, dy  = 2^\lambda\, \lambda \,  c.
\]
Note that the passage to the limit under the integral sign is justified by the boundedness of the functions  $|x|^{\lambda +1}\varphi(x)$ and $|x|^{\lambda +1}\omega(x)$. Also,  we applied condition \fer{norm3}.

Finally, in view of the fact that the L\'evy density $\omega(x) >0 $ in any set $|x| \le R$, with $R$ bounded positive constant, we have that, for any given $t>0$ both the functions $ |x|^{\lambda +1}f(x,t)$  and $|x|^{\lambda +1}\omega(x)$ are uniformly bounded from  above and below. 

Indeed, for $t \ge \delta>0$
 \[
 f(x,t) = f_{\alpha(t)}\ast\omega_{\beta(t)}(x) = \int_\R f_{\alpha(t)}(y)\omega_{\beta(t)} (x-y) \, dy \le \sup_x \omega_{\beta(t)}(x) \le
 \]
 \[
  \frac 1{\beta(t)}\sup_x \omega(x) = (1-e^{-t})^{-1/\lambda}\sup_x \omega(x) \le (1-e^{-\delta})^{-1/\lambda}\sup_x \omega(x) =C_\delta < +\infty. 
 \]
 Moreover
 \[
f(x,t) = \int_\R f_{\alpha(t)}(y)\omega_{\beta(t)} (x-y) \, dy \ge   \int_{\{ |y| \le R\}} f_{\alpha(t)}(y)\omega_{\beta(t)} (x-y) \, dy. 
 \]
 Therefore, for any positive constant $R$
 \[
\inf_{|x| \le R} f(x,t)  \ge  \inf_{|x| \le R}  \int_{\{ |y| \le R\}} f_{\alpha(t)}(y)\omega_{\beta(t)}(x-y) \, dy \ge  \inf_{|x| \le 2R} \omega_{\beta(t)}(x)\int_{\{ |y| \le R\}} f_{\alpha(t)}(y) \, dy.
 \]
 On the other hand, if $t \ge \delta$
  \[
  \inf_{|x| \le 2R} \omega_{\beta(t)}(x) = \frac 1{\beta(t)}   \inf_{|x| \le 2R/\beta(t)} \omega (x) \ge \inf_{|x| \le 2R(1-e^{-\delta})^{-1/\lambda}} \omega(x).
  \]
 Now, consider that a density function $f$ in the domain of normal attraction of the
$\lambda$-stable law, for any $\varsigma$  such that $0< \varsigma<\lambda$  satisfies \cite{Ibra71}
 \be\label{mome} \int_{\R}|x|^{\varsigma} \, f(x)\, dx <+\infty.
  \ee
 Thanks to \fer{mome}, since $\lambda >1$,
  \[
\int_{\{ |y| >R\}} f_{\alpha(t)}(y) \, dy  = \int_{\{ |y| >R/\alpha(t)\}} f(y) \, dy  \le  \frac{\alpha(t)}R \int_{\{ |y| >R/\alpha(t)\}} |y| f(y) \, dy \le m \frac{\alpha(t)}R ,
  \]
where  $m= \int_\R |x|f(x) \, dx < +\infty$.  Consequently, for each $t \ge \delta$ 
 \[
 \int_{\{ |y| >R\}} f_{\alpha(t)}(y) \, dy  \le  m \frac{e^{-\delta/\lambda}}R ,
 \]
 which implies
 \[
\inf_{|x| \le R} f(x,t)  \ge \left(1- m \frac{e^{-\delta/\lambda}}R \right) \inf_{|x| \le 2R(1-e^{-\delta})^{-1/\lambda}} \omega(x) = c_{\delta, R} >0.
 \] 
 Finally, we can choose $\delta >0$ and $R>0$ in such a way that, for $t \ge \delta$ the function $f(x,t)/\omega(x)$ is uniformly bounded in time from above and below in the domain $\{|x| \le R\}$. By virtue of \fer{attr}, we can in addition choose $R$   large enough to have $f(x,t)/\omega(x)$ uniformly bounded in time from above and below for $t \ge\delta$. Last, Chebyshev's inequality implies that for $t \ge \delta$ the solution density $f(x,t)$ is uniformly integrable. 

This is enough to guarantee that we can pass to the limit into the integral in the relative entropy. This concludes the proof of the lemma. 

\end{proof}

In the rest of this Section, we will assume that   the density $\varphi$ belongs to the  domain of normal attraction of $Z_\lambda$, with bounded relative fractional Fisher information $I_\lambda(\varphi)$.  With these hypotheses, let us evaluate the time evolution of the relative entropy $H(f(t)|\, \omega) $.  To this aim, let us rewrite the Fokker--Planck equation \fer{FFPP} in a way which is useful for our purposes. Since the L\'evy density (of order $\lambda$) satisfies the identity \fer{zz}, so that
 \[
\frac{ D_{\lambda -1}\, \omega }\omega = -\frac x\lambda,
 \]
we can write the fractional Fokker--Planck equation \fer{FFPP} in the form
 \be\label{FPFnew}
 \derpar{f}{t} = \derpar{}{x}\left[ f\left(\frac{D_{\lambda-1} f }f - \frac{ D_{\lambda -1}\, \omega }\omega\right) \right].
 \ee 
 Then, for all $t \ge \delta >0$  we obtain
 \begin{equations}\label{der-ent}
 & \frac d{dt} H(f(t)|\, \omega) =  \frac d{dt} \int_\R  f(x,t) \log \frac{f(x,t)}{\omega(x)} \, dx = \int_\R  \left( 1 +  \log \frac{f(x,t)}{\omega(x)} \right)\, \derpar{f}{t} \, dx = \\
 & \int_\R    \log \frac{f(x,t)}{\omega(x)} \, \derpar{f}{t} \, dx = \int_\R    \log \frac{f(x,t)}{\omega(x)} \, \derpar{}{x}\left[ f\left(\frac{D_{\lambda-1} f }f - \frac{ D_{\lambda -1}\omega }\omega\right) \right] \, dx  =\\
 & -  \int_\R  f\,  \left(\frac{ f '}f - \frac{ \omega' }\omega\right)  \left(\frac{D_{\lambda-1} f }f - \frac{ D_{\lambda -1}\omega }\omega\right) \, dx  = - \bar I_\lambda (f(t))\le 0.
  \end{equations}
 In \fer{der-ent}  we used the fact that the mass of the solution  is preserved along the evolution. Also, integration by parts is justified by the fact that, as proven in Lemma \ref{l11},  $\log [{f(x,t)}/{\omega(x)}] $ is uniformly  bounded in time for $t \ge \delta >0$,  and the  fractional Fisher information of $f(x,t)$ is bounded in view of Theorem \ref{bl}. Moreover, by Cauchy-Schwarz inequality, for any given density $f$ in the domain of attraction of the fractional Fisher information \fer{fish-r2}, the (nonnegative) entropy production $\bar I_\lambda (f)$ satisfies the bound
  \be\label{ep2}
  \bar I_\lambda (f) \le I(f)^{1/2} I_\lambda(f)^{1/2}.
   \ee
By virtue of Theorem \ref{bl}, as in \fer{ep}
  \[ 
  I_\lambda( f(t)) = I_\lambda (X_t) \le \alpha(t)^{2} I_\lambda(Y) = \alpha(t)^{2} I_\lambda(\varphi) ,
  \] 
 with $\alpha(t)$ given by \fer{con22}.  In addition, thanks to \fer{norm3}
  \[
  \max\{ \alpha(t)^\lambda, \beta(t)^\lambda\}  \ge \frac 12,
  \]
  which implies, by  inequality \fer{min2}
  \begin{equations}\label{BS3}
  & I(X_t) = I(\alpha(t)Y + \beta(t) Z) \le \min\{ I(\alpha(t) Y),\, I(\beta(t) Z_\lambda) = \\
  & \min\{ \alpha(t)^{-2} I(Z_\lambda),\, \beta(t)^{-2}I( Z_\lambda) \} \le  2^{2/\lambda} \min\{ I( Y),\, I(Z_\lambda) \}.
    \end{equations}
  Finally, for every time $t>0$ we have the bound
   \be
    \bar I_\lambda (f) \le \exp\{-t/\lambda\}\, 2^{1/\lambda} \, \min\{ I( \varphi),\, I(\omega) \}^{1/2}\, I_\lambda(\varphi)^{1/2}.
       \ee
 Therefore, if  the density $\varphi$ lies in the domain of attraction of the fractional Fisher information, and $\min\{ I( \varphi),\, I(\omega) \}$ is bounded, integrating \fer{der-ent} from zero to infinity, and recalling Lemma \ref{l11} we obtain the inequality
 \[
  H(\varphi |\, \omega) \le \lambda \, 2^{1/\lambda} \, \min\{ I( \varphi),\, I(\omega) \}^{1/2}\, I_\lambda(\varphi)^{1/2}.
  \]
We proved
\begin{thm}\label{main2} Let $X$ be a random variable with density $\varphi$ in the domain of normal attraction of the L\'evy symmetric random variable $Z_\lambda$, $1<\lambda <2$. If in addition $X$ has bounded Fisher information, and lies in the domain of attraction of the fractional Fisher information, the Shannon relative entropy $H(X|Z_\lambda)$ is bounded, and the following inequality holds
 \be\label{sob-frac}
  H(X |\, Z_\lambda) \le \lambda \, 2^{1/\lambda} \, \min\{ I( X),\, I(Z_\lambda) \}^{1/2}\, I_\lambda(X)^{1/2}. 
 \ee
\end{thm}

Inequality \fer{sob-frac} is the analogous of the logarithmic Sobolev inequality, which is obtained when $\lambda =2$ (Gaussian case). In this case,  the fractional Fisher information coincides with the classical Fisher information. As for the classical logarithmic Sobolev inequality, however, the inequality is saturated when the laws of $X$ and $Z_\lambda$ coincide. 

In short, let us take $\lambda =2$. Then, the steady state of the Fokker--Planck equation \fer{FFPP} is the Gaussian density $\omega_2$ defined in \fer{max} and equality \fer{der-ent}  takes the simple form
 \be\label{de-en}
 \frac d{dt} H(f(t)|\, \omega) = -I_2 (f(t)) = -I(f(t)|\, \omega _2),
  \ee
where $I(f|\omega _2)$ denotes the standard Fisher information defined in \fer{rel-f}.
In this case, however, formula \fer{ep} gives
 \[
  I(f(t)|\, \omega _2) \le e^{-t}  I(\varphi|\, \omega _2) .
 \]
 Hence, taking the integrals on both sides of \fer{de-en} from $0$ to $+\infty$, and making use of the previous inequality, shows that
  \be\label{lsi}
  H(\varphi |\, \omega) \le  I(\varphi|\, \omega _2),
  \ee
which is nothing but the logarithmic Sobolev inequality corresponding to the Gaussian density $\omega_2$ in one dimension \cite{Tos3}. The general case, corresponding to $\omega_\sigma$, with $\sigma \not=2$ follows easily from analogous argument. We remark that from inequality \fer{lsi} we can easily obtain the classical logarithmic Sobolev inequality in the form
 \[
 H(\varphi) + 1 +\frac 12 \log 4 \pi \le I(\varphi),
 \]
while in the L\'evy case, expressed by inequality \fer{sob-frac}, the boundedness of the right-hand side is guaranteed only for the relative fractional Fisher information. Hence we can not separate to arrive  to an inequality which connects the Shannon entropy of $\varphi$ to its fractional Fisher information which is unbounded. 

Nevertheless, by means of inequality \fer{sob-frac} one can pass convergence results to a stable law in terms of the fractional  Fisher information to convergence results in relative entropy.  Then, by the
 Csiszar--Kullback inequality \cite{Csi, Kul}, convergence in relative entropy  allows to
recover   convergence in $L^1(\R)$ from convergence in relative entropy.
  
\section{Convergence results in relative entropy} \label{regolare}

Let us consider the normalized sum
 \be\label{sum1}
T_n = \frac 1{n^{1/\lambda}}\sum_{j=1}^n X_j,
 \ee  
where the $X_j$'s are independent copies of a centered random
variable $X$ with density function $f$ which lies in the domain of normal attraction of the random variable $Z_\lambda$
 with { L\'evy symmetric stable density $\omega$}, and let $f_n$ denote the distribution function of $T_n$, $n\ge 1$. If in addition the density $f$ has bounded Fisher information, and belongs to the domain of attraction of the relative fractional Fisher information, so that $I_\lambda(f) < +\infty$,  inequality \fer{sob-frac} implies that the relative entropy $H(f |\, \omega)$ is bounded. Then, the same inequality guarantees that, for each $n >1$ 
 \be\label{start}
 H(T_n |\, Z_\lambda) \le \lambda \, 2^{1/\lambda} \, \min\{ I( T_n),\, I(Z_\lambda) \}^{1/2}\, I_\lambda(T_n)^{1/2}. 
  \ee 
 By virtue of Theorem \ref{main3}, inequality \fer{giu} implies that 
  \[
 I_\lambda(T_n)^{1/2} \le  \left( \frac 1n \right)^{(2-\lambda)/(2\lambda)} I_\lambda(X)^{1/2}.
  \]
Hence, convergence in relative entropy at the rate  $ n ^{-(2-\lambda)/(2\lambda)} $ follows as soon as the uniform boundedness of $(I(T_n)$ is proven.  The following theorem gives sufficient conditions to  ensure that $I(T_n) \le C$ for $n >1$, where  $C$ is a suitable constant. 

\begin{thm}\label{main-f} Let $f$ belong to the domain of normal attraction of the L\'evy symmetric random variable $Z_\lambda$, $1<\lambda <2$ and assume that there exists $M>0$ such that 
 \be\label{con23}
\int_\R |\hf(\xi)|^M(1+|\xi|^2)^k \, d\xi = C_M < +\infty.
 \ee
 Then, for $n \ge M/2$, $f_n \in H^k(\R)$. In addition, condition \fer{con23} holds with $M=2$ if $f \in H^k(\R)$, with $M > (2k+1)/\varepsilon$ if $ |\hf(\xi)||\xi|^\varepsilon$ is bounded for $|\xi| \ge 1$, where $\varepsilon >0$ is arbitrary, and with $M> 2k+1$ if $I(f)$ is bounded.
\end{thm}
  
\begin{proof}
Analogous result has been obtained in \cite{LT} in the case of the classical central limit theorem (cf. Theorem 5.1). The arguments there can be easily adapted to the present situation.
We take $ n \ge M$ and remark that 
 \[
 \hf_n(\xi) = \hf\left( \frac\xi{n^{1/\lambda}} \right)^n.
 \]
Hence, for any fixed $\delta >0$ we have
 \[
 \int_\R \left|  \hf\left( \frac\xi{n^{1/\lambda}} \right)\right|^{2n} (1+|\xi|^2)^k \, d\xi =
 \int_{|\xi| \le\, n^{1/\lambda}\delta} \left|  \hf\left( \frac\xi{n^{1/\lambda}} \right)\right|^{2n} (1+|\xi|^2)^k \, d\xi +
 \]
 \[
\int_{|\xi| >\, n^{1/\lambda}\delta}\left|  \hf\left( \frac\xi{n^{1/\lambda}} \right)\right|^{2n} (1+|\xi|^2)^k \, d\xi.
 \]
Consider now that, for $n>1$ 
 \[
\int_{|\xi| >\, n^{1/\lambda}\delta}\left|  \hf\left( \frac\xi{n^{1/\lambda}} \right)\right|^{2n} (1+|\xi|^2)^k \, d\xi =  
n^{1/\lambda} \int_{|\eta| >\, \delta}\left|  \hf\left( \eta \right)\right|^{2n} (1+ n^{2/\lambda}\,|\eta|^2)^k \, d\xi = 
 \]
 \[
 n^{1/\lambda} \int_{|\eta| >\, \delta}\left|  \hf\left( \eta \right)\right|^{2n} \left(\frac 1{n^{2/\lambda}}+ \,|\eta|^2\right)^k n^{2k/\lambda}\, d\xi \le
 \]
  \[
 n^{(2k+1)/\lambda} \sup_{|\eta| >\, \delta}\left|  \hf\left( \eta \right)\right|^{2n-M}\int_{|\eta| >\, \delta}\left|  \hf\left( \eta \right)\right|^{M} \left( 1+ \,|\eta|^2\right)^k n^{2k/\lambda}\, d\xi \le
 \]
 \[
 C_M  n^{(2k+1)/\lambda} \sup_{|\eta| >\, \delta}\left|  \hf\left( \eta \right)\right|^{2n-M}. 
 \]
In view of condition \fer{con23}, $f \in C(\R)$, and $|\hf(\eta)|<1$ for $\eta \not= 0$, with $  \sup_{|\eta| >\, \delta}\left|  \hf\left( \eta \right)\right| <1$. Consequently, as $n \to +\infty$
 \[
 n^{(2k+1)/\lambda} \sup_{|\eta| >\, \delta}\left|  \hf\left( \eta \right)\right|^{2n-M} \to 0.
 \]
Next, recall that $f(x)$ belongs to the domain of normal attraction of the L\'evy stable law of order $\lambda$. This implies that $\hf(\xi)$ satisfies the condition \cite{Ibra71}
 \be\label{nor5}
 1 - \hf(\xi) = (1-R(\xi))|\xi|^\lambda,
 \ee
where $R(\xi) \in L^\infty(\R)$, and $|R(\xi)| = o(1)$ as $\xi \to 0$. 
Thanks to \fer{nor5}, for each $\gamma >0$, we can choose $\delta$ small enough in such a way that, when $|\xi| < \delta$
 \[ 
 \hf(\xi) = 1 - (1-R(\xi))|\xi|^\lambda \le \exp\left\{ -|\xi|^\lambda(1-\gamma)\right\}.
 \]
Thus we obtain
 \[
 \int_{|\xi| \le\, n^{1/\lambda}\delta} \left|  \hf\left( \frac\xi{n^{1/\lambda}} \right)\right|^{2n} (1+|\xi|^2)^k \, d\xi  \le
 \]
 \[
 \int_{|\xi| \le\, n^{1/\lambda}\delta} \left[ \exp\left\{ -\frac{|\xi|^\lambda}n(1-\gamma)\right\}\right]^{2n} (1+|\xi|^2)^k \, d\xi  =
 \]
\[
 \int_{|\xi| \le\, n^{1/\lambda}\delta} \exp\left\{ -2|\xi|^\lambda(1-\gamma)\right\} (1+|\xi|^2)^k \, d\xi  \le
 \]
\[
 \int_\R\exp\left\{ -2|\xi|^\lambda(1-\gamma)\right\} (1+|\xi|^2)^k \, d\xi  = C(k,\lambda,\gamma).
 \]
This shows in particular that $f_n$ is bounded in $H^k(\R)$,  and in addition
 \[
 \lim\sup_n \|f_n\|^2_{H^k(\R)} \le \left\|\exp\left\{ -|\xi|^\lambda(1-\gamma)\right\}\right\|^2_{H^k(\R)}
 \]
for all $\gamma >0$. Taking $\gamma$ to zero, we conclude that we have
 \be\label{okk}
 \lim\sup_n \|f_n\|^2_{H^k(\R)} \le \left\| \omega \right\|^2_{H^k(\R)}.
 \ee

 Next, by the definition of $H^k(\R)$, \fer{con23} holds with $M=2$ if $f \in H^k(\R)$. If $|\hf(\xi)||\xi|^\varepsilon$ is bounded for $|\xi| \ge 1$ and for some $\varepsilon >0$, then for some constant $C>0$ we have that
 \[
|\hf(\xi)|^M(1+|\xi|^2)^k \le C |2\xi|^{2k-M\varepsilon}, \qquad |\xi| \ge 1. 
 \]
 Therefore \fer{con23} holds if $M\varepsilon > 2k+1$.
 
 Finally, if $I(f)< +\infty$, then $g= \sqrt f \in H^1(\R)$. Thus $\hf= \hg\ast \hg$, where $\hg$ satisfies
  \[
  \int_\R |\hg(\xi)|^2(1+ |\xi|^2) \, d\xi < +\infty.
  \]
 Then, for all $\xi \in \R$ we have
 \[
 |\hf(\xi)||\xi| \le \int_\R \hg(\xi-\eta)\hg(\eta) (|\xi - \eta| + |\eta|) \, d\eta \le
 \]
 \[
2 \left( \int_\R |\hg(\xi)|^2 |\xi|^2 \, d\xi \right)^{1/2} \left( \int_\R |\hg(\xi)|^2  \, d\xi \right)^{1/2}.
 \]
At this point, we can apply the previous remark with $\varepsilon =1$.
 \end{proof}  

If condition \fer{con23} of Theorem \ref{main-f} holds true for $k=1$, we conclude that the sequence $f_n(x)$ is uniformly bounded in $H^1(\R)$, and consequently uniformly bounded by imbedding in $W^{1,1}(\R)$. Hence, if $I(f) < +\infty$, for each $n>3$   there exists a positive constant $C_{TV}$ such that
 \[
\| f_n\|_{TV} =  \int_\R |f'_n(x)| \, dx \le C_{TV}.
 \]
Consequently, the sequence $\{f_n\}_{n>3}$ is uniformly bounded in total variation norm. By Proposition 3.2 in \cite{BCG3}, if independent variables $Y_j$, $j= 1,2,3$ have densities $p_j(x)$ of bounded variation, then $Y_1+Y_2+Y_3$ has finite Fisher information, and moreover
 \[
 I(Y_1+Y_2+Y_3) \le \frac 12 \left( \|p_1\|_{TV}  \|p_2\|_{TV} + \|p_1\|_{TV}  \|p_3\|_{TV} + \|p_2\|_{TV}  \|p_3\|_{TV}\right).
 \]
This allows to conclude that, provided $I(f)< +\infty$, for all $n\ge 1$, $I(T_n) \le C$.

We conclude with the following

\begin{thm}\label{main-entr} Let the  random variable $X$  belong to the domain of normal attraction of the random variable $Z_\lambda$ with { L\'evy symmetric stable density $\omega$}. If in addition the density $f$ of $X$ has bounded Fisher information, and belongs to the domain of attraction of the relative fractional Fisher information, so that $I_\lambda(f) < +\infty$, the sequence of density functions $f_n$ of the normalized sums $T_n$,  defined in \fer{sum1} converge to zero in relative entropy and 
 \be\label{decay3}
 H(T_n |\, Z_\lambda) \le  C_{\lambda}(X) \, \left( \frac 1n \right)^{(2-\lambda)/(2\lambda)} I_\lambda(X)^{1/2}.
 \ee
\end{thm}

Despite the fact that the constant $ C_{\lambda}(X) $ is not explicitly known, Theorem \ref{main-entr} insures convergence to zero in relative entropy at a precise rate. Thanks to  Csiszar--Kullback inequality \cite{Csi, Kul}, convergence in relative entropy  implies   convergence in $L^1(\R)$ at the sub-optimal rate $n^{-(2-\lambda)/(4\lambda)}$ . Indeed, if $f$ satisfies all the conditions of Theorem \ref{main-entr}
 \be\label{ck}
 \|f _n- \omega\|_{L^1(\R)}^2 \le 2 H(T_n |\, Z_\lambda) \le  C_{\lambda}(X) \, \left( \frac 1n \right)^{(2-\lambda)/(2\lambda)} I_\lambda(X)^{1/2}.
 \ee
 On the other hand, convergence in $L^1(\R)$ of $f_n$ to $\omega$ implies that $f_n$ converges to $\omega$ weakly in $H^k(\R)$. 
This weak convergence, combined with inequality \fer{okk} yields the strong convergence in $H^k(\R)$.  This implies the following

\begin{cor}\label{main-h} 
Let $f$ satisfy the conditions of Theorem \ref{main-entr}. Then $f_n$ converges to $\omega$ in $H^k(\R)$ for all $k \ge 0$. Moreover, there is convergence of $f_n$ to $\omega$ in the homogeneous Sobolev space $\dot H^k(\R)$ at the rate $[n^{-(2-\lambda)/(4\lambda)}]^{2/(2k+3)}$.
\end{cor}

\begin{proof}
For a given function in $L^1(\R)\cap \dot H^{k+1}(\R)$, for any given constant $R >0$ it holds
 \[
 \int_\R |\xi|^{2k}|\widehat f(\xi)|^2 \, d\xi  \le \int_{|\xi|<R} |\xi|^{2k}|\widehat f(\xi)|^2 \, d\xi + \frac 1{R^2}\int_{|\xi|\ge R} |\xi|^{2k+2}|\widehat f(\xi)|^2 \, d\xi \le
 \]
 \[
\|f\|_{L^1}^2 \int_{|\xi|<R} |\xi|^{2k}\, d\xi + \frac 1{R^2}\|f\|_{\dot H^{k+1}}^2 = \frac{2R^{2k+1}}{2k+1}\|f\|_{L^1}^2 + \frac 1{R^2}\|f\|_{\dot H^{k+1}}^2.
 \]
Optimizing over $R$ we get
 \be\label{def4}
 \|f\|_{\dot H^{k}} \le C_k \|f\|_{L^1}^{2/(2k+3)} \|f\|_{\dot H^{k+1}}^{(2k+1)/(2k+3)},
 \ee 
where $C_k$ is an explicitly computable constant. Applying \fer{def4} to $\f_n- \omega$  we then obtain
 \[
\|f_n-\omega \|_{\dot H^{k}} \le C_k \|f_n-\omega \|_{L^1}^{2/(2k+3)}\left( \|f_n\|_{\dot H^{k+1}}+  \|\omega\|_{\dot H^{k+1}}\right)^{(2k+1)/(2k+3)}. 
 \]
Now, consider that the sequence $f_n$, for $n\ge n_0$ sufficiently large, is bounded in  $H^k(\R)$. From this and \fer{ck} the result follows.
\end{proof}

\section{Conclusions}
In a recent paper \cite{To-Fish}, in connection with the study of the normalized sums of random variables in the central limit theorem for stable laws, we  introduced the definition of relative fractional Fisher information. This nonlocal functional is based on a suitable modification of the  linear score function used in theoretical statistics. As the
linear score function $f'(X)/f(X)$ of a random variable $X$ with a (smooth) probability density $f$ identifies  Gaussian variables as the
unique random variables for which the score is linear (i.e. $f'(X)/f(X) =CX$), L\'evy symmetric stable laws are identified as the unique random variables for which the new defined fractional linear score is linear. In this paper we showed that the fractional Fisher information can be fruitfully used to bound the relative (to the L\'evy stable law) Shannon entropy, through an inequality similar to the classical logarithmic Sobolev inequality. Analogously to the central limit theorem, where monotonicity of entropy along the sequence provides an explicit rate of convergence to the Gaussian law for some smooth densities \cite{JB}, in the case of the central limit theorem for stable laws we succeeded in proving convergence in $L^1(\R)$ at explicit rate, and, for smooth densities, convergence in various Sobolev spaces (still with rate).

\section{Appendix}\label{App}
In this short appendix we will collect the notations concerning fractional derivatives, with some applications to the L\'evy stable laws. Also, we list  the various functional spaces used in the paper. Last, we recall some properties of Linnik distribution, that represents the main example of probability density to which the results of the present paper can be applied.

\subsection{\it Fractional derivatives}

For $0 <\nu < 1$ we let $R_\nu$ be the one-dimensional
\emph{normalized} Riesz potential operator defined for locally integrable functions by
\cite{Rie, Ste}
 \[
 R_\nu(f)(x) = S(\nu) \int_\R \frac{f(y)\, dy}{|x-y|^{1-\nu}}.
 \]
The constant $S(\nu)$ is chosen to have
 \be\label{rt}
\widehat{R_\nu(f)}(\xi) = |\xi|^\nu \widehat f(\xi).
 \ee
Since for $0 <\nu < 1$ it holds  \cite{Lieb83}
 \be\label{hom}
 \mathcal F |x|^{\nu -1} = |\xi|^{-\nu} \pi^{1/2} \Gamma \left(\frac{1-\nu}2 \right) \Gamma \left(\frac{\nu}2 \right),
 \ee
where, as usual $\Gamma(\cdot)$ denotes the Gamma function, the value of $S(\nu)$
is given by
 \[
 S(\nu) = \left[\pi^{1/2} \Gamma \left(\frac{1-\nu}2 \right) \Gamma \left(\frac{\nu}2 \right)\right]^{-1}.
 \]
 Note that $S(\nu) = S(1-\nu)$.

We then define the fractional derivative of order $\nu$ of a real function $f$ as
($0 <\nu < 1$)
 \be\label{fa}
 \frac{d^\nu f(x)}{dx^\nu} = \mD_\nu f(x) = \frac{d}{dx}R_{1-\nu}(f)(x).
 \ee
Thanks to \fer{rt}, in Fourier variables
 \be\label{d1}
 \widehat{{\mD}_\nu f}(\xi) = i \frac{\xi}{|\xi|} |\xi|^\nu \widehat{f}(\xi).
 \ee
It is immediate to verify that, for $0<\nu<1$,   L\'evy centered stable laws $\omega(x)$ of parameter $\lambda = 1+\nu$ satisfy
 \be\label{new-l}
 \frac{\mD_\nu\, \omega(x)}{\omega(x)} = -\frac x{1+\nu}.
 \ee
Indeed, identity \fer{new-l}
is verified  if and only if, on the set $\{f>0\}$
 \be\label{ccw}
\mD_\nu f(x) = -\frac{xf(x)}{1+\nu}.
 \ee
Passing to Fourier transform, this identity yields
 \[
i \xi |\xi|^{\nu-1} \widehat f(\xi) = -i \frac 1{1+\nu}\frac{\partial \widehat f(\xi)}{\partial
\xi},
 \]
and from this follows
 \be\label{ok3}
\widehat f(\xi) = \widehat f(0) \exp\left\{-|\xi|^{\nu +1}\right\}.
 \ee
Finally,   imposing that $f(x)$ is a probability
density function (i.e. by fixing $\widehat f(\xi= 0) =1$), we obtain  that the L\'evy
stable law of order $1 + \nu$ is the unique probability density solving  \fer{new-l}.

\subsection{\it Functional framework}

We list below the various functional spaces used in the paper.
      For $p \in [1, +\infty)$ and $q \in [1, +\infty)$, we denote by $L_q^p$  the weighted Lebesgue spaces
            \begin{equation*}
                L^p_q := \left\{ f : \R \rightarrow \R \text{ measurable; }\|f\|_{L^p_q}^p := \int_\R |f(x)|^p \, (1+ x^2)^{q/2} \, dx < \infty \right\}.
            \end{equation*}
             In particular, the usual Lebesgue spaces are given by
            \[ L^p := L^p_0.\]
            Moreover, for $f \in L^1_q$, we can define for any $\kappa \leq q$  the $\kappa^{th}$ order \emph{moment} of $f$ as the quantity
            \begin{equation*}
              m_\kappa(f):= \int_\R f(x) \, |x|^{\kappa} dx\, < \, \infty.
            \end{equation*}
            For $k \in \N$, we denote by $W^{k,p}$ the Sobolev spaces
            \begin{equation*}
                W^{k,p} := \left\{ f \in L^k; \|f\|_{W{k,p}}^p := \sum_{|j| \leq k} \int_{\R} \left |f^{(j)}(v)\right |^p \, dx < \infty \right \}.
            \end{equation*}
          If $p=2$ we set $H^k := W^{k,2}$.

          Given a probability density $f$, we define its \emph{Fourier transform} $\mathcal F(f)$ by
        \begin{equation*}
          \mathcal F(f)(\xi) = \widehat{f}(\xi) := \int_\R e^{- i \, \xi \, x} f(x)\, dx, \qquad \forall  \xi \in \R.
        \end{equation*}
        The Sobolev space $H^k$ can equivalently be defined for any $k \geq 0$ by the norm
            \begin{equation*}
                \|f\|_{H^{k}} := \left \| \mathcal F\left (f \, \right )\right \|_{L^2_{2k}}.
            \end{equation*}
            The homogeneous Sobolev space $\dot H^k$ is then defined by the homogeneous norm
            \begin{equation*}
              \|f\|_{\dot H^k}^2 := \int_\R |\xi|^{2 k} \left | \widehat{f}(\xi) \right |^2 \, d \xi.
            \end{equation*}

We moreover denote by $\| f\|_{TV}$ the total variation norm of the probability density function $f$
 \[
\| f\|_{TV} = \sup \sum_{j=1}^n |f(x_j) -f(x_{j-1})|, 
 \]
 where the supremum ranges over all finite collections $x_1<x_2 < \cdots < x_n$. For a (sufficiently regular) function the total variation norm can also be represented as
  \[
\| f\|_{TV} = \int_\R |f'(x)|\, dx. 
 \]
\bigskip

\subsection{\it Linnik distribution}
The leading
example of a function which belongs to the domain of attraction of the
$\lambda$-stable law is the so-called {\it Linnik distribution} \cite{L, L2},
expressed in Fourier variable by
  \[
  \widehat p_\lambda(\xi) = \frac 1{1+ |\xi|^\lambda}.
  \]
For all $0<\lambda \leq 2$,  this function   is the
characteristic function of a symmetric probability distribution. In addition, when
$\lambda > 1$, $\widehat p_\lambda \in L^1(\R)$, which, by applying the
inversion formula, shows that $p_\lambda$ is a probability density function.

Owing to the explicit expression of its Fourier transform, it is immediate to verify that Linnik distribution of parameter $\lambda$ satisfies condition \fer{con23} with 
\[
M> 2\, \frac{k+1}\lambda.
\]
Other properties of Linnik's distributions can be extracted from its representation
as a mixture (cf. Kotz and Ostrovskii \cite{KO}). For any given pair of positive
constants $a$ and $b$, with $0 < a < b \le 2$ let $g(s, a, b)$ denote the probability
density
 \[
g(s, a, b) = \left( \frac b\pi \sin\frac{\pi a}b \right) \frac{s^{a -1}}{1+ s^{2a} +
2s^a \cos\frac{\pi a}b }, \quad 0 <s<\infty.
 \]
Then, the following equality holds \cite{KO}
 \be\label{mix}
 \widehat p_a(\xi) = \int_0^\infty \widehat p_b(\xi/s) g(s, a, b)\, ds,
 \ee
or, equivalently
 \[
 p_a(x) = \int_0^\infty p_b(sx) g(s, a, b)\, ds.
 \]
This representation allows us to generate Linnik distributions of different parameters
starting from a convenient base, typically from the Laplace distribution
(corresponding to $b =2$). In this case, since $\widehat p_2(\xi)= 1/(1+|\xi|^2)$
(alternatively $p_2(x) = e^{-|x|}/2$ in the physical space), for any $\lambda$ with
$1<\lambda <2$ we obtain the explicit representation
 \be\label{oo}
  \widehat p_\lambda(\xi) = \int_0^\infty \frac {s^2}{s^2+ |\xi|^2} \,\,g(s, \lambda, 2)\, ds,
 \ee
or, in the physical space
 \be\label{lin2}
 p_\lambda(x) = \int_0^\infty \frac s2 \, e^{-s|x|} g(s, \lambda, 2)\, ds.
  \ee
Owing to \fer{lin2} we obtain easily that, for $1<\lambda<2$,  Linnik's probability
density is a symmetric and bounded function, non-increasing and convex for $x >0$.

In \cite{To-Fish} we used representation formula \fer{lin2} to prove that Linnik distribution belongs to the domain of attraction of the fractional Fisher information. 

Moreover, by virtue of \fer{lin2}, differentiating under the integral sign when $x>0$, by the dominated convergence theorem we  obtain
 \[
 |p_\lambda'(x)|=  \left| - \int_0^\infty \frac x{|x|}\, \frac{s^2}2\, e^{-s|x|} g(s, \lambda, 2)\, ds \right| \le  \int_0^\infty  \frac{s^2}2\, e^{-s|x|} g(s, \lambda, 2)\, ds.
 \]
Since for $0<\delta < 1$
 \[
 s e^{-|x|s} = s e^{-|x|\delta s} e^{-|x|(1-\delta)s} \le \frac 1{\delta |x| e} e^{-|x|(1-\delta)s},
 \]
it holds
 \[
|p_\lambda'(x)| \le  \frac 1{\delta |x| e} p_\lambda((1-\delta)x).
 \]
In particular, since Linnik distribution belongs to the domain of attraction of the stable law, conditions \fer{dof} imply 
 \[
 \lim_{x \to \pm \infty}|x| \frac{|p_\lambda'(x)|}{p_\lambda(x)} \le \frac 1{\delta e} \, \lim_{x \to \pm \infty} \frac{p_\lambda((1-\delta)x)}{p_\lambda(x)} = \frac {(1-\delta)^{1+\lambda}}{\delta e}
 \]
Hence, as $|x| \to \infty$, ${|p_\lambda'(x)|}/{p_\lambda(x)} \sim 1/|x|$, so that
 \be\label{lin44}
\frac{|p_\lambda'(x)|}{p_\lambda(x)} \le \frac c{1+|x|}, 
 \ee
for $x \in \R$, with some positive constant $c$. This  shows that Linnik distribution has a bounded Fisher information. 

Finally, Linnik distribution satisfies all hypotheses of Theorem \ref{main-entr}, and, provided that the random variables $X_j$'s  are independent copies of a centered random variable $X$ with a Linnik distribution, the relative entropy of the sequence $T_n$ converges to zero at a rate $n^{-(2-\lambda)/(2\lambda)}$.

\bigskip \noindent

\noindent{\bf Acknowledgments:} This work has been written within the activities of
the National Group of Mathematical Physics of INDAM (National Institute of High
Mathematics). The support of the  project ``Optimal mass transportation, geometrical
and functional inequalities with applications'', financed by the Minister of
University and Research, is kindly acknowledged. 





\end{document}